\documentclass[12pt]{amsart}

\usepackage[T1]{fontenc}
\usepackage[utf8]{inputenc}
\usepackage[a4paper,margin=1in]{geometry}
\usepackage{amsmath,amssymb}
\usepackage{nicefrac}
\usepackage{xcolor}
\usepackage{tikz}
\usepackage{comment}
\usepackage{thmtools}
\usepackage{thm-restate}
\usepackage{enumitem}
\usepackage{microtype}
\usepackage{hyperref}

\hypersetup{
  hidelinks,
  pdftitle={Alon\textendash Tarsi for hypergraphs},
  pdfauthor={Marcin Anholcer; Bart{\l}omiej Bosek; Grzegorz Gutowski;
             Micha{\l} Laso{\'n}; Jakub Przyby{\l}o; Oriol Serra;
             Micha{\l} Tuczy{\'n}ski; Llu{\'{\i}}s Vena; Mariusz Zaj{\k{a}}c},
  pdfsubject={Bounds on the Alon\textendash Tarsi number of a hypergraph
              polynomial in terms of the edge density of the hypergraph.
              },
  pdfkeywords={Alon\textendash Tarsi number; Combinatorial Nullstellensatz;
               hypergraph polynomial; edge density; list coloring;
               paintability; 1-2-3 Conjecture},
  pdflang={en},
}

\renewcommand{\leq}{\leqslant}
\renewcommand{\le}{\leqslant}
\renewcommand{\geq}{\geqslant}
\renewcommand{\ge}{\geqslant}

\def\N{\mathbb N}
\def\F{\mathbb F}
\def\C{\mathbb C}
\def\ed{\mathrm{ed}}

\theoremstyle{plain}
\newtheorem{theorem}{Theorem}
\newtheorem{lemma}[theorem]{Lemma}
\newtheorem{corollary}[theorem]{Corollary}
\newtheorem{proposition}[theorem]{Proposition}
\newtheorem{claim}[theorem]{Claim}

\theoremstyle{definition}
\newtheorem{definition}[theorem]{Definition}
\newtheorem{example}[theorem]{Example}
\newtheorem{conjecture}[theorem]{Conjecture}

\title{Alon--Tarsi for hypergraphs}

\author{Marcin Anholcer}
\address{Institute of Informatics and Quantitative Economics, Pozna\'n University of Economics and Business, Pozna\'n, Poland}
\email{marcin.anholcer@ue.poznan.pl}
\thanks{Marcin Anholcer was supported by the National Science Centre, Poland, grant no.~2023/51/B/HS4/00829.}

\author{Bart{\l}omiej Bosek}
\address{Institute of Theoretical Computer Science, Faculty of Mathematics and Computer Science, Jagiellonian University, Krak\'ow, Poland}
\email{bartlomiej.bosek@uj.edu.pl}

\author{Grzegorz Gutowski}
\address{Institute of Theoretical Computer Science, Faculty of Mathematics and Computer Science, Jagiellonian University, Krak\'ow, Poland}
\email{grzegorz.gutowski@uj.edu.pl}
\thanks{Bart{\l}omiej Bosek, Grzegorz Gutowski, and Mariusz Zaj{\k{a}}c were supported by the National Science Centre, Poland, grant no.~2019/35/B/ST6/02472.}

\author{Micha{\l} Laso{\'n}}
\address{Institute of Mathematics of the Polish Academy of Sciences, Warszawa, Poland}
\email{michalason@gmail.com}
\thanks{Micha{\l} Laso{\'n} was supported by the National Science Centre, Poland, grant no.~2019/34/E/ST1/00087.}

\author{Jakub Przyby{\l}o}
\address{AGH University of Krakow, Faculty of Applied Mathematics, Al. A. Mickiewicza 30, 30-059 Krakow, Poland}
\email{jakubprz@agh.edu.pl}
\thanks{Jakub Przyby{\l}o's research was partially supported by the AGH University of Krakow under grant no. 16.16.420.054 and Excellence Initiative -- Research University AGH University of Krakow, funded by the Polish Ministry of Science and Higher Education.}

\author{Oriol Serra}
\address{Department of Mathematics, Universitat Polit\`ecnica de Catalunya, Barcelona, Spain}
\email{oriol.serra@upc.edu}

\author{Micha{\l} Tuczy{\'n}ski}
\address{Faculty of Mathematics and Information Science, Warsaw University of Technology, Warsaw, Poland}
\email{michal.tuczynski@pw.edu.pl}

\author{Llu{\'{\i}}s Vena}
\address{Department of Mathematics, Universitat Polit\`ecnica de Catalunya, Barcelona, Spain}
\email{lluis.vena@upc.edu}
\thanks{Oriol Serra and Llu{\'{\i}}s Vena were supported by the Grant PID2023-147202NB-I00 funded by MICIU/AEI/10.13039/501100011033.}

\author{Mariusz Zaj{\k{a}}c}
\address{Faculty of Mathematics and Information Science, Warsaw University of Technology, Warsaw, Poland}
\email{mariusz.zajac@pw.edu.pl}


\begin{document}

\begin{abstract}
Given a hypergraph $H=(V, E)$, define for every edge $e\in E$ a linear expression with arguments corresponding to the vertices. Next, let the polynomial $p_H$ be the product of such linear expressions for all edges. Our main goal is to find a relationship between the Alon--Tarsi number of $p_H$ and the edge density of $H$. We prove that $AT(p_H)=\lceil \ed(H)\rceil+1$ if all the coefficients in $p_H$ are equal to $1$ and the base field has characteristic zero. Our main result is that, over an arbitrary field, if on every edge the coefficients are not all equal, then they can be permuted within the edges so that for the resulting polynomial $p_H^\prime$, $AT(p_H^\prime)\leq 2\lceil \ed(H)\rceil+1$ holds. We conjecture that this bound holds for every hypergraph polynomial without permuting its coefficients. If this were true, then in particular a significant generalization of the famous 1-2-3 Conjecture would follow.
\end{abstract}

\maketitle

\section{Introduction}

Let $G=(V, E)$ be a graph with the vertex set $V=[n]$. For an orientation $\vec{G}$ of the edges of $G$, the \emph{graph polynomial} of $\vec{G}$ is defined as the homogeneous multivariate polynomial $p_{\vec{G}}(x_1, \ldots, x_n)=\prod_{(i,j)\in E(\vec{G})}(x_j-x_i)$. One of the applications of this polynomial arises from the fact that $p_{\vec{G}}(a_1, \ldots, a_n)\neq 0$ for some $(a_1, \ldots, a_n)\in \N^n$ if and only if the coloring $\chi (i)=a_i$ is a proper coloring of $G$. Alon and Tarsi \cite{AT1992} developed an algebraic tool, which has become known as the Combinatorial Nullstellensatz (see Alon \cite{A1999}), to ensure that such a nonzero value of $p_{\vec{G}}$ can be found in a Cartesian product $A_1\times \cdots \times A_n$, providing an upper bound on the list chromatic number of $G$. If the expansion of $p_{\vec{G}}$ in some base field $\mathbb{F}$ is
$$
p_{\vec{G}}(x_1, \ldots, x_n)=\sum_{\alpha} c_{\alpha}x_1^{\alpha_1}\cdots x_n^{\alpha_n},
$$
where $\alpha=(\alpha_1, \ldots, \alpha_n)$ runs over $\N_0^n$, then the parameter of interest for the application of the Combinatorial Nullstellensatz, called the \emph{Alon--Tarsi number} of $G$ (see Jensen and Toft \cite{JT1995}) is 
$$
AT(G)=\min_{\alpha: c_{\alpha}\neq 0}\max \{\alpha_1, \ldots, \alpha_n\}+1.
$$
The Combinatorial Nullstellensatz ensures that $AT(G)$ is an upper bound for the list chromatic number of $G$ (for different proofs and generalizations of Combinatorial Nullstellensatz see \cite{Michalek, S2010, Lason}).  

The same framework has been extended to hypergraphs in the following way. Let $H=(V,E)$ be a (multi)hypergraph, where $E$ is a family of subsets of $V=[n]$. Every polynomial of the form
\begin{equation}\label{eq:hp}
p_H(x_1, \ldots ,x_n)=\prod_{e\in H}(\sum_{i\in e} a_{e,i}x_i),
\end{equation}
where $a_{e,i}\neq 0$ are taken from a base field $\mathbb{F}$, is called a {\it hypergraph polynomial} of $H$. A hypergraph polynomial is homogeneous of degree $|E|$. Following the above notation, if 
$$
p_{H}(x_1, \ldots, x_n)=\sum_{\alpha} c_{\alpha}x_1^{\alpha_1}\cdots x_n^{\alpha_n},
$$
is the expansion of the polynomial in $\mathbb{F}$, we call the Alon--Tarsi number of the polynomial $p_H$ the value
$$
AT(p_H)=\min_{\alpha: c_{\alpha}\neq 0}\max \{\alpha_1, \ldots, \alpha_n\}+1.
$$
Given a hypergraph $H$, its \emph{chromatic number} $\chi(H)$ is the minimum number of colors that allows a vertex coloring without monochromatic edges of size at least two. One can also consider two generalizations of the proper colorings defined this way. Assume that every vertex of $H$ is assigned a list of at most $k$ colors. If for any such list assignment one can pick a color from each list so that the resulting vertex coloring has no edge of size at least $2$ that is monochromatic, then we say that $H$ is $k$-choosable. The smallest $k$ for which $H$ is $k$-choosable is called the list chromatic number of $H$ and denoted $\chi_L(H)$. In another variant of the problem (defined for graphs by Schauz \cite{SchauzPaintCorrect}) the lists are not known in advance, but in a step-by-step procedure it is announced which vertices have a certain color, say $c$, on their lists; the vertices to be colored $c$ need to be selected immediately and no recoloring is possible in the future. If there is a selection strategy such that all the vertices can be colored from their lists, we say that $H$ is $k$-online choosable (or $k$-paintable). The smallest $k$ for which $H$ is $k$-online choosable is called online choosability number (or paintability number) and denoted $\chi_P(H)$. Note that for every hypergraph $H$ we have $\chi(H)\leq \chi_L(H)\leq \chi_P(H)$. 

Ramamurthi and West \cite{RW2005} obtained a generalization of the Alon--Tarsi theorem \cite[Corollary 1.4]{AT1992}, relating Eulerian orientations of a graph with the list chromatic number, to $k$-uniform hypergraphs (whose all edges contain $k$ vertices) for $k$ being a prime. To this end, they use a hypergraph polynomial over $\C$ where the coefficients of the variables are the $k$-th roots of unity. A similar framework was used by Schauz \cite{S2010} to obtain upper bounds on the online choosability number of $k$-uniform and $k$-partite hypergraphs. Hefetz \cite{H2011} introduced an extension of the Combinatorial Nullstellensatz which can also be applied to non-necessarily uniform hypergraphs. All these results were reformulated and presented in a unified form in the monograph of Zhu and Balakrishnan \cite[Section 3.9]{ZhuBal}. It is worth mentioning that the list colorability of hypergraphs has also been investigated using methods other than the Combinatorial Nullstellensatz, in particular by variants of the Probabilistic Method \cite{AlonSpencer}. To be more specific, the hypergraphs with high chromatic number were studied by Wang and Qian \cite{WanQ}. Lower bounds on the list chromatic number of simple hypergraphs were studied by Saxton and Thomason \cite{SaxTho} and by Alon and Kostochka \cite{AlonKos}. The same authors \cite{AlonKos2} analyzed the lower bounds for dense uniform hypergraphs, Haxell and Verstraete \cite{HaxVer} focused on the regular $3$-uniform hypergraphs, while M\'eroueh and Thomason \cite{MerTho} on the random $r$-partite hypergraphs. The list chromatic number of Steiner Triple Systems was analyzed by Haxell and Pei \cite{HaxPei}. Wang, Qian, and Yan \cite{WQY} investigated the number of list colorings of uniform hypergraphs.


Our motivation comes from the study of the Alon--Tarsi number of planar graphs. Zhu \cite{Z2019} proved that the Alon--Tarsi number of planar graphs is at most five, giving an alternate proof of the celebrated theorem by Thomassen \cite{T1994} of the $5$-choosability of this class of graphs, as well as their $5$-paintability. 

Our main goal is to obtain an upper bound on the Alon--Tarsi number of a hypergraph polynomial over an arbitrary field $\mathbb{F}$ in terms of the edge density of the hypergraph. Recall that the edge density of a hypergraph $H=(V,E)$ is defined as 
$$
{\rm ed}(H)=\max_{\emptyset\neq X\subseteq V} \frac{|E(X)|}{|X|},
$$
where $E(X)=\{e\in E(H): e \subseteq X\}$ denotes the set of edges of the subhypergraph $H[X]$ of $H$ induced by $X\subseteq V$.

Let us briefly explain why the edge density is a natural parameter in this context. Every monomial in the expansion of a hypergraph polynomial of $H$ arises by choosing one variable $x_{r(e)}$ from each bracket, that is, by choosing a representative $r(e)\in e$ of every edge, and the exponent of $x_v$ in the resulting monomial is exactly the number of edges represented by $v$. Since all the edges contained in a set $X$ must be represented inside $X$, some vertex of $X$ represents at least $|E(X)|/|X|$ of them, so no choice of representatives can do better than $\lceil \ed(H)\rceil$; conversely, Hall's theorem guarantees that a system of representatives in which no vertex is used more than $\lceil \ed(H)\rceil$ times always exists. The edge density is thus precisely the obstruction that pure counting can see. What remains is to ensure that the monomial produced in this way has a nonzero coefficient, and the two results below achieve this in different ways: for a fully balanced polynomial that coefficient is a positive integer and therefore survives in characteristic zero (Lemma~\ref{obs:fullybal-den}), whereas for a fully unbalanced one the coefficients are permuted inside the edges so as to force it to be nonzero (Theorem~\ref{thm:main}).

Consider a hypergraph polynomial of $H=([n],E)$
$$
p(x_1, \ldots, x_n)=\prod_{e\,\in\, E} \big( \sum_{i\,\in\,e} a_{e,i}\,x_{i} \big) \in \F[x_1, \ldots, x_n]
$$
over $\F$. We call an edge $e\in E$ \emph{balanced in $p$} if $a_{e,i}=a_{e,j}$ for all $i,j\in e$, and \emph{unbalanced in $p$} otherwise; when $p$ is clear from the context we simply say that $e$ is balanced or unbalanced. The polynomial $p$ is then called {\it fully unbalanced} if every edge of $H$ is unbalanced in $p$. 

For permutations of edges $\sigma_{e_1}\!\in S_{e_1},$ $\sigma_{e_2}\!\in S_{e_2},$ $\ldots,$ $\sigma_{e_m}\!\in S_{e_m}$, the \emph{permuted polynomial} is defined as 
\[
p_{\sigma_{e_1}\sigma_{e_2}\ldots\sigma_{e_m}}(x_1, \ldots, x_n)=\prod_{i\in[m]} \big( \sum_{v\,\in\,e_i} a_{e_i,\sigma_{e_i}(v)}\,x_{v} \big) \in \F[x_1, \ldots, x_n].
\]
In other words, it is constructed from the original polynomial by permuting the coefficients (but not the variables) in every bracket.

The main result of this paper is the following theorem.
\begin{restatable}{theorem}{thmMain}\label{thm:main}
For every hypergraph $H=(V,E)$ and fully unbalanced hypergraph polynomial $p$ of $H$ there exist permutations of edges $\sigma_{e_1}\!\in S_{e_1},$ $\sigma_{e_2}\!\in S_{e_2},$ $\ldots,$ $\sigma_{e_m}\!\in S_{e_m}$ such that
$$
AT(p_{\sigma_{e_1}\sigma_{e_2}\ldots\sigma_{e_m}}) \leqslant 2 \cdot \lceil {\rm ed}(H) \rceil +1.
$$
\end{restatable}

The proof reduces the hypergraph to a multigraph. Using the system of representatives described above, every hyperedge $e$ is shrunk to a pair $f\subseteq e$ containing the representative of $e$, in such a way that the resulting multigraph still has density at most $\lceil \ed(H)\rceil$. Every edge of a multigraph has only two endpoints, so its degeneracy is at most twice its density, and this is exactly where the factor $2$ in the bound comes from. Both proofs of Theorem~\ref{thm:main} given below start from this reduction; in the first of them the permutations of the coefficients are then used to guarantee that the monomial provided by it does not vanish.

We note that, for every collection of permutations $\sigma=(\sigma_e:e\in E)$ and every hypergraph polynomial $p$ of a hypergraph $H$, $p_{\sigma}$ as defined in the statement of Theorem \ref{thm:main} is again a hypergraph polynomial of $H$. In particular, the bound in the Theorem also provides a bound on the list chromatic number (choosability) and the online list chromatic number (paintability) of the hypergraph (see Section \ref{sec:app}). The bound is tight, as, for example, the Fano plane has edge density $1$ and chromatic number $3$.

Observe that the Alon--Tarsi number of a polynomial may change when permuting the coefficients within the edges, as illustrated in the following examples.

\begin{example}  Consider the tetrahedron hypergraph $H$, with the vertex set $V(H)=[4]$ and the edge set $E(H)=\{e_1=\{1,2,3\},e_2=\{1,2,4\},e_3=\{1,3,4\},e_4=\{2,3,4\}\}$. For the hypergraph polynomial
\begin{align*}
p_H(x_1,x_2,x_3,x_4)=&(x_1+\omega x_2+\omega^2 x_3)(x_1+\omega x_2+\omega^2 x_4)\\
&(x_3+\omega x_1+\omega^2 x_4)(x_2+\omega x_3+\omega^2 x_4),
\end{align*}
where $\omega=\frac{1}{2}(-1+i\sqrt{3})$ is a third root of unity, $AT(p_H)=3$. Indeed, one can verify that $\mathbf{coeff}(p_H,x_1x_2x_3x_4)=0$, so $AT(p_H)\geq 3$, while e.g. $\mathbf{coeff}(p_H,x_1^2x_2^2)=-1-i\sqrt{3}\neq 0$. However, if one swaps the coefficients to obtain the polynomial
\begin{align*}
p^\prime_H(x_1,x_2,x_3,x_4)=&(x_1+\omega x_2+\omega^2 x_3)(x_1+\omega x_2+\omega^2 x_4)\\
&(x_1+\omega x_3+\omega^2 x_4)(x_2+\omega x_3+\omega^2 x_4),
\end{align*}
it comes out that  $AT(p^\prime_H)=2$, since $\mathbf{coeff}(p^\prime_H,x_1x_2x_3x_4)=-3+3\sqrt{3}i\neq 0$.
\end{example}

\begin{example}
    A smaller example can be exhibited by allowing less regularity in the coefficients. Consider the hypergraph  polynomial on the complete graph $K_3$ 
$$
p_H(x_1,x_2,x_3)=(x_1+ x_2)(2 x_2+ x_3) (x_3 - 2 x_1).
$$
One can verify that  $AT(p_H)= 3$, because $\mathbf{coeff}(p_H,x_1x_2x_3)=0$ and $\mathbf{coeff}(p_H,x_1x_3^2)=1 \neq 0$.  However, for the polynomial 
$$
p^\prime_H(x_1,x_2,x_3)=(x_1+ x_2)(x_2+ 2 x_3) (x_3 - 2 x_1)
$$
obtained by swapping coefficients in $p_H$, we have $\mathbf{coeff}(p^\prime_H,x_1x_2x_3)=-3\neq 0$ and in consequence $AT(p^\prime_H)=2$.
\end{example}

We conjecture that the statement of Theorem \ref{thm:main} can be extended to any hypergraph polynomial:

\begin{restatable}{restatableconjecture}{conMain}\label{conj:1} For every hypergraph polynomial $p$ of a hypergraph $H=(V,E)$,
$$
AT(p)\le 2 \cdot \lceil {\rm ed}(H) \rceil +1.
$$
\end{restatable}

The truth of Conjecture \ref{conj:1} would provide, as observed by Grytczuk \cite{GrytczukPersonal}, a simple alternate proof of the longstanding 1-2-3 Conjecture of Karo{\'n}ski, {\L}uczak and Thomason, recently confirmed by Keusch \cite{K2023}, even in its list version, which is currently open.  We discuss this topic and further applications of our main result in Section \ref{sec:app}. 

The paper is organized as follows. In Section \ref{sec:prel} we give more detailed definitions and provide a preliminary bound on the Alon--Tarsi number of hypergraph polynomials. Section \ref{sec:main} contains the proof of Theorem \ref{thm:main}. Section \ref{sec:app} discusses some applications of the Alon--Tarsi number of hypergraph polynomials. We conclude the paper with some final remarks, including open problems.

\section{Preliminaries}\label{sec:prel}

Let $\mathbb{F}$ be a field and let
$$
p(x_1,\ldots, x_n)=\sum_{\alpha} c_{\alpha} x_1^{\alpha_1}\cdots x_n^{\alpha_n}
$$
be a polynomial in $\mathbb{F}[x_1, \ldots, x_n]$, where the sum is over all sequences $\alpha=(\alpha_1, \ldots, \alpha_n)$ with non-negative integer entries. Let 
$$
S(p)=\{\alpha: \sum_{i}\alpha_i=\text{deg}\; p, c_{\alpha}\neq 0\},
$$
denote the set of sequences corresponding with the monomials of maximum degree appearing in $p$. The {\it Alon--Tarsi number} of $p$ is
$$
AT(p)=\min_{\alpha\in S(p)}\max \,\{\alpha_1, \ldots, \alpha_n\} +1.
$$

Let $H=(V, E)$ be a hypergraph with set of edges $E$, where each edge $e\in E$ is a nonempty subset of $V$. We allow multiple edges in $H$ but still refer to it as a hypergraph instead of a multihypergraph, and emphasize with the word {\it simple} if multiple edges are not allowed. However, we do not consider loops: all edges are of size at least $2$. We consider finite hypergraphs, and we will always assume that the set of vertices is $V=[n]=\{1,2, \ldots, n\}$, if not stated otherwise. 

We recall that a hypergraph polynomial over $\F$ of $H$ is any polynomial of the form
$$
p(x_1, \ldots, x_n)=\prod_{e\,\in\, E} \big( \sum_{i\,\in\,e} a_{e,i}\,x_{i} \big) \in \F[x_1, \ldots, x_n],
$$
for some choice of coefficients $a_{e,i}\neq 0$. If for a given polynomial $p$ there is a hypergraph $H$ for which $p$ is a hypergraph polynomial of $H$, then we will call $p$ a hypergraph polynomial.
Recall that an edge $e\in E$ is balanced in $p$ if $a_{e,i}=a_{e,j}$ for all $i,j\in e$, and unbalanced in $p$ otherwise.
The hypergraph polynomial $p$ is {\it fully unbalanced} if every edge is unbalanced in $p$, and {\it fully balanced} if every edge is balanced in $p$; in the latter case the coefficients can all be assumed to be equal to one up to scalar multiplication. We denote the fully balanced hypergraph polynomial of $H$ by $\overline{p_H}$, namely
$$
\overline{p_H}(x_1, \ldots, x_n)=\prod_{e\,\in\, E} \big( \sum_{i\,\in\,e} x_{i} \big) \in \F[x_1, \ldots, x_n].
$$
Our main objective is to relate the Alon--Tarsi number of a hypergraph polynomial to the hypergraph's edge density, as stated in Theorem \ref{thm:main}. For each nonempty subset $X\subseteq V$, we denote by $H[X]$ the subhypergraph of $H$ induced by the vertices in $X$, and by $E(X)=E(H[X])$ its set of edges. The edge density of $H$ is
$$
\ed (H)=\max_{\emptyset\neq X\subseteq V}\frac{|E(X)|}{|X|}.
$$

The edge density provides a natural lower bound for the  Alon--Tarsi number of hypergraph polynomials.

\begin{lemma}\label{rem:edlowerbound}
For every hypergraph polynomial $p$, we have
\begin{equation*}\label{eq:edlowerbound} AT(p)\ge \lceil \ed(H) \rceil +1.
\end{equation*}
\end{lemma}
\begin{proof}
Let us take any monomial $\prod_{e\in E} x_{r(e)}$ in the expansion of $p$, where $r(e)$ is a variable (vertex) chosen from the bracket corresponding to the edge $e$.
Let $k':=\max_{j\in V}|r^{-1}(j)|$ denote the maximum power of a variable in $\prod_{e\in E}x_{r(e)}$. Let us show that $k'\geqslant {\rm ed}(H)$. This will complete the proof because $k'$ is an integer and  represents
the maximum power in any monomial, in particular in the one that minimizes this value.

Let us take any $X\subseteq V$.
Because $r(e)\in e$ for each $e \in E$ we have $r(E(X))\subseteq X$. Hence, $E(X)\subseteq r^{-1}(X)$ and so
$$
|E(X)|\leqslant |r^{-1}(X)|= |\bigcup_{i\in X} r^{-1}(i)|\leqslant \sum_{i\in X}|r^{-1}(i)| \leqslant k'\cdot |X|.
$$
Thus, ${\rm ed}(H)=\max_{\emptyset \neq X\subseteq V}\frac{|E(X)|}{|X|}\leqslant k'$. This completes the proof.
\end{proof}

In the case of fully balanced polynomials over fields of characteristic zero, the opposite inequality is also true.

\begin{lemma}\label{obs:fullybal-den} 
For the fully balanced hypergraph polynomial $\overline{p_H}$ over a field of characteristic zero of a hypergraph $H$  we have
\[
AT(\overline{p_H}) = \lceil {\rm ed}(H) \rceil +1.
\]
\end{lemma}

\begin{proof} Observe that $\overline{p_H}$ is a sum of monomials of the form
$$
\prod_{e\in H} x_{r(e)},
$$
where $r:E\to V$ is a function satisfying $r(e)\in e$ for all $e\in E$, whose coefficients are positive integers.

By Lemma~\ref{rem:edlowerbound} we only need to prove that $AT(\overline{p_H}) \leqslant \lceil {\rm ed}(H) \rceil +1$.
Let us denote $k:=\lceil {\rm ed}(H) \rceil$.
For any subset $\emptyset\neq Y\subseteq V$, since
$$
\frac{|E(Y)|}{|Y|} \leqslant \max_{\emptyset\neq X\subseteq V} \frac{|E(X)|}{|X|} = {\rm ed}(H) \leqslant \lceil {\rm ed}(H) \rceil = k,
$$
we have $$|E(Y)| \leqslant k\cdot|Y|.$$
Based on the hypergraph $H=(V,E)$, let us define the hypergraph $$H'=(V'=V\times [k],\,E'=\{e\times [k]: e\in E\}),$$ that is, every vertex $v$ of $H$ is split into $k$ copies $(v,1),\ldots,(v,k)\in V'$, and each edge $e'=e\times [k] \in E'$ contains all clones of the vertices in $e \in E$. Note that $|E|=|E'|$ and $|V'|=k\cdot |V|$.

Let us show that $H'$ admits a system of representatives by verifying that Hall's condition holds. Let
$
\{e_{1}',\ldots,e_{s}'\}\subseteq E'
$
be any family of edges in $H'$.
Then
\begin{align*}
|\{e_{1}',\ldots,e_{s}'\}|&=|\{e_{1},\,\ldots,\,e_{s}\}|\\
&\le |E(e_1\cup\ldots \cup e_s)|\\
&\le k\cdot |e_1\cup\ldots \cup e_s|\\
&=|(e_1\times [k])\cup\ldots \cup (e_s\times [k])|
=|e'_1\cup \ldots \cup e'_s|.
\end{align*}
Therefore, by Hall's theorem, there is an injection $r': E' \rightarrow V'$ such that $r'(e')\in e'$ for each $e'\in E'$.
Let us define the function $r:E\rightarrow V$ by setting $r(e):=j$ if $r'(e')=(j,l)$ for some $l\in[k]$.
Because $|r^{-1}(j)|\leqslant k$ for every $j\in V$, and the characteristic of $\F$ is $0$, the monomial $x_{r(e_1)}\cdot x_{r(e_2)}\cdot \ldots \cdot x_{r(e_m)}$ is a witness to $AT(\overline{p_H})\leqslant k+1$, which completes the proof.
\end{proof}

\begin{example}
Note that in the proof we used the fact that the field $\F$ has characteristic $0$, which implies that the coefficient of the monomial $x_{r(e_1)}\cdot x_{r(e_2)}\cdot \ldots \cdot x_{r(e_m)}$ does not vanish. The lemma does not need to be true in the case of a finite field. For example, in the case of the field $\mathbb{F}_2$ and $H=C_n$ being the cycle on $n$ vertices, we have $AT(\overline{p_H}) =3> \lceil {\rm ed}(H) \rceil +1$.
\end{example}

Let us define the \emph{degeneracy} of a hypergraph $H=(V,E)$ as
\begin{equation}\label{eq:deg-first}
\delta(H):=\max_{X\subseteq V} \min_{i\in X} d_{H[X]}(i),
\end{equation}
where $d_{H[X]}(i)=|\{e\in E: i\in e \subseteq X \}|$ is the degree of vertex $i$ in $H[X]$.

Similarly to the case of graphs, it is easy to show that the degeneracy of a hypergraph $H=(V,E)$, where $V=[n]$, can be equivalently defined as
\begin{equation}\label{eq:deg-second}
\delta(H) = \min_{\sigma\in S_n} \max _{i\in[n]} |\{e \in E(\{\sigma(1), \ldots, \sigma (i)\}): \sigma(i) \in e \}|,
\end{equation}
where the minimum runs over all the permutations of $[n]$. As in the case of graphs, we can also relate density and degeneracy in hypergraphs.
\begin{lemma}\label{lem:deg-den}
For every hypergraph $H=(V,E)$ we have
$$
\delta(H) \leqslant \max_{e\in E} |e| \cdot {\rm ed}(H).
$$
\end{lemma}
\begin{proof}
Let $Y\subseteq V$ be a witness to the degeneracy value of $\delta(H)$, that is
$$
\delta(H) = \min_{i\in Y} d_{H[Y]}(i).
$$
Now we double count the set of pairs $\{(i,e)\in Y\!\times\! E : i\in e \subseteq Y\}$, that is:
\begin{align*}
 \delta(H) \cdot |Y| &\leqslant \sum_{i\in Y} |\{(i,e) \in Y \!\times\! E : i\in e\subseteq Y\}|\\
 &=\sum_{e\in E(Y)} |\{(i,e)\in Y \!\times\! E : i\in e\subseteq Y\}|\\
 &\leqslant \max_{e\in E} |e| \cdot |E(Y)| 
\leqslant\max_{e\in E} |e| \cdot {\rm ed}(H) \cdot |Y|.
 \end{align*}
\end{proof}

\begin{theorem}\label{thm:max-e}
For every hypergraph polynomial $p$ of the hypergraph $H=([n],E)$ we have
\[
AT(p)-1 \leqslant \delta(H) \leqslant \max_{e\in E} |e| \cdot {\rm ed}(H) \leqslant \max_{e\in E} |e| \cdot (AT(p)-1).
\]
\end{theorem}

\begin{proof} Let $p$ be the polynomial
$$
p(x_1, \ldots, x_n)=\prod_{e\in E} \left(\sum_{i\in e} a_{e,i}x_i\right),
$$
where $a_{e,i}\in \F\setminus \{0\}$ for each pair $(e,i)$.

The second and third inequalities follow from Lemmas \ref{lem:deg-den} and  \ref{rem:edlowerbound}, respectively.
The only thing left to prove is that
\begin{equation}\label{eq:ATq-leq-deg}
AT(p) \leqslant \delta(H)+1.
\end{equation}
To prove inequality (\ref{eq:ATq-leq-deg}), we use the equivalent definition of hypergraph degeneracy from (\ref{eq:deg-second}).
To simplify the notation, without loss of generality, we can assume that the permutation $\sigma\in S_n$ giving the minimum in (\ref{eq:deg-second}) is the identity, that is
\begin{equation*}
\delta(H) = \max _{i\in[n]} |\{e \in E: i \in e\subseteq \{1, \ldots, i\}\}|.
\end{equation*}
Equivalently, 
\begin{equation}\label{eq:deg-second-max}
\delta(H) = \max _{i\in[n]} |\{e \in E: \max (e) = i \} |.
\end{equation}
As a witness of inequality (\ref{eq:ATq-leq-deg}), we use the monomial 
$$
c \prod_{e\in E} x_{\max(e)},
$$
which is one of the monomials resulting from the choice of variables in the product of the brackets constituting the definition of the polynomial $p(x_1, \ldots, x_n)=\prod_{e\in E} \left( \sum_{i\,\in\,e} a_{e,i}\,x_{i} \right)$, and show that $c\neq 0$.
First, we note that Equation (\ref{eq:deg-second-max}) implies that the highest power of a single variable in the monomial $c \prod_{e\in E} x_{\max(e)}$ does not exceed $\delta(H)$.
Next, we claim that the sequence of choices
$$
(a_{e,\max(e)}: e\in E)
$$
from the appropriate brackets $(\left(\sum_{i\in e}a_{e,i}x_i): e\in E\right)$
is the only possible one that generates the monomial $c \prod_{e\in E} x_{\max(e)}$.
As a consequence, we will get that $c=\prod_{e\in E} a_{e,\max(e)} \neq 0$, which completes the proof.

We will prove the uniqueness of the choice of factors by induction on $n$.
For $n=0$, this is trivially fulfilled.
Now assume that $n>0$ and that for $n-1$ the induction thesis is satisfied.

Since the last variable $x_{n}$ must be selected from all brackets in which it appears, then the variables $x_{1},\ldots,x_{n-1}$ are selected from those brackets that do not contain $x_{n}$. Therefore, the coefficient $c$ of $\prod_{e\in E}x_{e,\max (e)}$ equals
$$
c=c'\prod_{e\in E: n\in e} a_{e,n},
$$
where $c'$ is the coefficient of $\prod_{e\in E([n-1])}x_{e,\max (e)}$ in the polynomial $\prod_{e\in E([n-1])}\left(\sum_{i\in e} a_{e,i}x_i\right)$. The claim follows by induction. This completes the proof.
\end{proof}

\section{Main result}\label{sec:main}

It must be admitted that the above theorems and observations are more or less a generalization of our knowledge about graphs.
The real goal of the research is to try to replace the coefficient $\max_{e\in E}|e|$ in Theorem~\ref{thm:max-e} with a constant independent of the hyperedge size.
We consider the main contribution of our work to be Theorem~\ref{thm:main}, which, in a sense, accomplishes this goal, although it unfortunately uses some manipulation in the shape of a polynomial.

Recall that the main result of this work is as follows.
\thmMain*

%
For the clarity of presentation, we divide the proof of Theorem \ref{thm:main} into two lemmas.

\begin{lemma}\label{lem:H-G}
Let $H=(V,E)$ be a hypergraph with density ${\rm ed}(H)\leqslant k$ for some integer $k$ and the set of edges $E=\{e_1,\ldots,e_m\}$ such that $|e_i|\geqslant 2$ for each $i=1,2,\ldots,m$.

There exists a multigraph $G=(V,F)$ of density ${\rm ed}(G)\leqslant k$ with the set of edges $F=\{f_1,\ldots,f_m\}$ satisfying $f_i \subseteq e_i$ for each $i=1,2,\ldots,m$.
\end{lemma}

\begin{proof}
It is enough to consider the case $k=\lceil {\rm ed}(H) \rceil$. As in the proof of Lemma ~\ref{obs:fullybal-den}, define the hypergraph $H'=(V\times [k], E'=\{e\times [k]: e\in E\})$ in which, as shown in that proof, there is a system of representatives $r':E'\rightarrow V\times [k]$ with $r'(e')\in e'$ for each $e'\in E'$ and $\{r'(e'):e'\in E'\}$ pairwise distinct. 

We define the multigraph $G=(V,F)$ as follows. Let $r:E\to V$ be defined as $r(e)=v$ if $r'(e')=(v,s)\in V\times [k]$ for some $s\in [k]$. For every edge $e_i\in E$ define the edge $f_i\in F$ formed by the pair $\{r(e_i),j\}$ for some arbitrarily chosen $j\in e_i\setminus r(e_i)$, hence indeed $f_i\subseteq e_i$; we will abuse the notation above, and from now on we regard $r$ as a function $r: F\to V$ where $r(f_i)$ corresponds to $r(e_i)$. Note that $|F|=|E|$. 

What remains to prove is that ${\rm ed}(G)\leqslant k$.

Let us consider any nonempty subset $X\subseteq V$.
Because $r(f)\in f$ for each $f\in F$ we have $r(F(X))\subseteq X$ and in consequence $F(X)\subseteq r^{-1}(X)$. Hence,
$$
|F(X)|\leqslant |r^{-1}(X)|= \left|\bigcup_{i\in X} r^{-1}(i)\right|\leqslant \sum_{i\in X}|r^{-1}(i)| \leqslant k\cdot |X|.
$$
The last inequality follows because $r'$ is an injection and the projection from $V'$ to $V$ is $ k$-to-$1$. Thus ${\rm ed}(G)=\max_{\emptyset\neq X\subseteq V}\frac{|F(X)|}{|X|}\leqslant k$,
which completes the proof.
\end{proof}

One might notice that Lemma \ref{lem:H-G} could be also proved using the fact that $ed(H)\leq k$ is equivalent to the statement that there is a mapping $r: E(H) \to V(H)$ such that $|r^{-1}(i)| \leq k$ for each vertex $i$, which essentially stems from the proof of Lemma \ref{obs:fullybal-den}.

Thus, Lemma \ref{lem:H-G} allows one to associate with a hypergraph $H$ a multigraph $G$ with the properties stated in the Lemma. We observe that a coloring $\chi:V \to \N$ of a hypergraph $H$ is proper if $\chi$ is proper in the associated multigraph $G$, although the reverse implication is not true.
\begin{example}
Consider a hypergraph with $V=[3]$, $E=\{e=V\}$ and $r(e)=1$. Let the edge set of $G$ be $F=\{f=\{1,2\}\}$. So, we have a hypergraph with one edge containing all its vertices, with representative $1$, and a graph with one edge whose other endpoint is vertex $2$. The density condition is obviously satisfied, and there is a coloring $\chi(1)=\chi(2)=1, \chi(3)=2$ which is proper for $H$ but not for $G$.
\end{example}

In particular, we obtain the following Corollary bounding the list chromatic number and the paintability number of $H$.

\begin{corollary} \label{cor:deg} The list chromatic number and the paintability of a hypergraph $H$ satisfy
$$
\chi_L(H)\leq \chi_P(H)\leq 2\,\lceil {\rm ed}(H) \rceil+1.
$$ 
\end{corollary}

\begin{proof} Simply observe that the list chromatic number and the paintability of the associated multigraph $G$ is upper bounded by its degeneracy plus one, see e.g. \cite[p. 63]{ZhuBal}, and from Lemma \ref{lem:deg-den} we have $\delta (G)\le 2\,{\rm ed}(G)$.
\end{proof}

Note that the bound on the list chromatic number can be improved for certain special families of hypergraphs. For instance, Cherkashin and Gordeev \cite[Theorem 6]{cherkashin} proved that for any $2$-colorable hypergraph $H$, $\chi_L(H)\leq \,\lceil {\rm ed}(H) \rceil+1$. 

The following Lemma is the second ingredient used in the proof of Theorem \ref{thm:main}.

\begin{lemma}\label{lem:main}
Let $H=([n],E)$ be a hypergraph with the set of edges $E=\{e_1,e_2,\ldots,e_m\}$ such that $|e_i|\geqslant 2$ for each $i\in[m]$.

Assume that there is a system of representatives $r:E \rightarrow V$ such that 
$$
r(e_i) < \max(e_i) \;\text{ for all } i\in[m].
$$
Then, for a fully unbalanced hypergraph polynomial $p$ of $H$, there exist permutations of the edges $\sigma_{e_1}\!\in\!S_{e_1},$ $\sigma_{e_2}\!\in\!S_{e_2},$ $\ldots,$ $\sigma_{e_m}\!\in\!S_{e_m}$ such that the coefficient of the monomial $x_{r(e_1)}\cdot x_{r(e_2)} \cdot \ldots \cdot x_{r(e_{m})}$ in the polynomial $p_{\sigma_{e_1} \sigma_{e_2} \ldots \sigma_{e_m}}$ is nonzero.
\end{lemma}

\begin{proof}
Without loss of generality, we can assume that the hyperedges in $E$ are numbered so that
\begin{equation}\label{eq-e1-e2-em}
\max (e_1) \leqslant \max (e_2) \leqslant \ldots \leqslant \max (e_m).
\end{equation}
The proof is by induction on $n+m$. The Lemma trivially holds for $m=0$ and every $n\ge 0$ (in such a case $p\equiv 1$). Assume $m>0$, so that $n>0$. We may assume that $\max(e_m)=n$ as otherwise we may apply induction on the hypergraph $(V \setminus \{n\}, E)$. 

Let us assume that the polynomial $p$ is of the form
$$
p(x_1, \ldots, x_n)=\prod_{i\in [m]} \big( \sum_{j\,\in\,e_i} a_{e_i,j}\cdot x_{j} \big).
$$
We refer to the inductive assumption for 
\begin{enumerate}
\item the hypergraph $H':=(V, E \,\setminus\, \{e_m\})$,
\item the system of representatives
$r':{E \,\setminus\, \{e_m\}}\rightarrow V$ such that $r'(e_i):=r(e_i)$ for all $i = 1,\ldots, m-1$, and
\item the hypergraph polynomial $$p'(x_1, \ldots, x_n):=\prod_{i\in [m-1]} \big( \sum_{j\,\in\,e_i} a_{e_i,j}\cdot x_{j} \big).$$
\end{enumerate}
By the induction hypothesis, there exist permutations of the edges $\sigma_{e_1}\!\in\!S_{e_1},$ $\sigma_{e_2}\!\in\!S_{e_2},$ $\ldots,$ $\sigma_{e_{m-1}}\!\in\!S_{e_{m-1}}$ such that the coefficient of the monomial $x_{r(e_1)}\cdot x_{r(e_2)} \cdot \ldots \cdot x_{r(e_{m-1})}$ in the polynomial $ p'_{\sigma_{e_1} \sigma_{e_2} \ldots \sigma_{e_{m-1}}}$ is nonzero.

Now consider the following polynomial for the hypergraph $H$ with the choice of $\sigma_{e_m}$ the identity:
\begin{align*}
p_{\sigma_{e_1}\ldots\sigma_{e_{m-1}}\sigma_{e_m}}(x_1, \ldots, x_n)
&=\left( \sum_{j\,\in\,e_m} a_{e_m,j}\cdot x_{j} \right)\cdot p'_{\sigma_{e_1}\ldots\sigma_{e_{m-1}}}(x_1, \ldots, x_n)\\
&=\sum_{j\,\in\,e_m} a_{e_m,j}\cdot x_{j}\cdot p'_{\sigma_{e_1}\ldots\sigma_{e_{m-1}}}(x_1, \ldots, x_n).
\end{align*}
Let $b_j$ be the coefficient of $x_{r(e_1)}\cdots x_{r(e_m)}/x_j$ in $p'_{\sigma_{e_1}\ldots\sigma_{e_{m-1}}}$ if $j\in e_m\cap r(E)$ and $b_j=0$ otherwise. The above decomposition of the polynomial $p_{\sigma_{e_1}\ldots\sigma_{e_{m-1}}\sigma_{e_m}}$ shows that the coefficient $c$ of the monomial $x_{r(e_1)} \cdots x_{r(e_m)}$ in $p_{\sigma_{e_1}\ldots\sigma_{e_{m-1}}\sigma_{e_m}}$ is also decomposed analogously:
$$
c = \sum_{j\in e_m} a_{e_m,j} \cdot b_j.
$$
If $c \neq 0$, taking $\sigma_{e_m} = \textrm{id}\in S_{e_m}$ completes the proof.
So we can assume that $c = 0$.
If there are two vertices $k,l \in e_m$ such that
\begin{equation}\label{eq:uw}
a_{e_m,k}\neq a_{e_m,l}\quad \text{and}\quad b_k \neq b_{l},
\end{equation}
then, by choosing $\sigma_{e_m}=(kl)\in S_{e_m}$ as the transposition of $k$ and $l$ instead of the identity, the coefficient of $x_{r(e_1)}\cdots x_{r(e_m)}$ in $p_{\sigma_{e_1}\ldots \sigma_{e_m}}$ is 
$$
c' = \sum_{j\in e_m} a_{e_m,\sigma_{e_m}(j)} \cdot b_j,
$$
which is different from zero, because
\begin{align*}
c' = c'-c& = \sum_{j \in e_m} a_{e_m,\sigma_{e_m}(j)} \cdot b_j - \sum_{j \in e_m} a_{e_m,j} \cdot b_j\\
&=a_{e_m,k} \cdot b_l + a_{e_m,l} \cdot b_k - a_{e_m,k} \cdot b_k - a_{e_m,l} \cdot b_l\\
&=(a_{e_m,k}-a_{e_m,l})(b_l- b_k)\neq 0.
\end{align*}
Here the last inequality uses only the facts that both factors are nonzero and that a field has no zero divisors.

We shall find the desired $k,l$ satisfying \eqref{eq:uw} in the set $\{j,j',r(e_m), n\}\subseteq e_m$ where $j,j'$ are witnesses that $p$ is fully unbalanced for the bracket resulting from the edge $e_m$, that is, $a_{e_m,j}\neq a_{e_m,j'}.$ We use the following simple claim.

\begin{claim}\label{claim1}
Let $\{a,b,c,d\}$ be a multiset and $f,g:\{a,b,c,d\}\to \F$. Assume that $f(a)\neq f(b)$ and $g(c)\neq g(d)$. Then there is a pair $x,y\in \{a,b,c,d\}$ such that $f(x)\neq f(y)$ and $g(x)\neq g(y)$.
\end{claim}

\begin{proof}[Proof of Claim \ref{claim1}.] By contradiction. Suppose that $f(x)\neq f(y)$ implies $g(x)=g(y)$ (or equivalently $g(x)\neq g(y)$ implies $f(x)=f(y)$). By the assumptions, we have $a\neq b$ and $c\neq d$. Color the edge $xy$ of the complete graph $K_4$ with the vertices labeled with the members of $\{a,b,c,d\}$ red if $f(x)\neq f(y)$ and $g(x)=g(y)$ and blue if $g(x)\neq g(y)$ and $f(x)=f(y)$. By the assumptions, $ab$ is a red edge and $cd$ is a blue edge. At least one of the edges $ad$ or $bd$ must be red by the definition of the coloring. Assume without loss of generality that $ad$ is red. Then $g(a)=g(b)=g(d)$ and the edges $ca$ and $cb$ are blue. But then $f(a)=f(c)=f(b)$, contradicting the hypothesis.
\end{proof}

We apply the above Claim to $\{a,b,c,d\}=\{j,j',r(e_m), n\}$ and the functions $f(x)=a_{e_m,x}$ and $g(x)=b_x$. We observe that $g(r(e_m))=b_{r(e_m)}$ is the coefficient of $x_{r(e_1)}\cdot \ldots \cdot x_{r(e_{m-1})}$ in the polynomial $p'_{\sigma_{e_1}\ldots\sigma_{e_{m-1}}}$ which, by the inductive hypothesis, is nonzero. On the other hand, since $r(e_1), \ldots, r(e_m)<n$ by the assumption of the Lemma, the variable $x_n$ does not appear in $x_{r(e_1)}\cdots x_{r(e_m)}$ and therefore $b_n=0$. Thus we have $g(r(e_m))\neq g(n)$, while $f(j)\neq f(j')$ by the choice of $j$ and $j'$. By the claim, there are $k,l\in \{j,j',r(e_m),n\}$ satisfying \eqref{eq:uw}. By the argument preceding the Claim, this completes the proof.
\end{proof}

We can now proceed to the proof of the main theorem.

\begin{proof}[Proof of Theorem \ref{thm:main}]
Consider a hypergraph $H=([n],E)$ with the set of edges $E=\{e_1,\ldots,e_m\}$, where $|e_i|\ge 2$ for each $i=1, \ldots, m$.
It follows from Lemma \ref{lem:H-G} that there exists a multigraph $G=(V,F)$ of density ${\rm ed}(G)\leqslant \lceil {\rm ed}(H) \rceil$ with the set of edges $F=\{f_1,\ldots,f_m\}$ such that $f_i \subseteq e_i$ for each $i=1,2,\ldots,m$.
In turn, from Lemma \ref{lem:deg-den} we get the following bound on the degeneracy of the multigraph $G$:
\begin{equation*}
\delta(G) \leqslant 2 \cdot {\rm ed}(G) \leqslant 2 \cdot \lceil {\rm ed}(H) \rceil.
\end{equation*}
Degeneracy defined in (\ref{eq:deg-second}), where the ``left-going'' edges are counted, can be represented equivalently by counting the ``right-going'' edges in the following way:
$$
\delta(G) = \min_{\sigma\in S_n} \max _{i\in[n]} |\{j \in [m]: {\sigma(i)} \in f_j \subseteq \sigma([n]\setminus [i-1])\}|.
$$

We can assume, without loss of generality, that the permutation realizing the minimum is the identity.
Then,
\begin{align}\label{eq-ed-minf}
2 \cdot \lceil {\rm ed}(H) \rceil \geqslant \delta(G) 
&=\max _{i\in[n]} |\{j \in [m]: i\in f_j\subseteq [n]\setminus [i-1] \}|\nonumber\\
&= \max _{i\in[n]} |\{j \in [m]: \min (f_j) = i \}|.
\end{align}

Let us define the representative system $r:E\rightarrow V$ for a hypergraph $H$ in such a way that $r(e_j) := \min(f_j)$ for each $j=1,\ldots,m$.
This is well defined because $f_j \subseteq e_j$ for each $j=1,\ldots,m$.
As a consequence, we get
$$
r(e_j) = \min(f_j) < \max(f_j) \leqslant \max(e_j).
$$
This allows us to use Lemma \ref{lem:main} for the hypergraph $H$, the system of representatives $r$, and the polynomial $p$ of $H$.
Then, we get that there exist permutations of the edges $\sigma_{e_1}\!\in\!S_{e_1},$ $\sigma_{e_2}\!\in\!S_{e_2},$ $\ldots,$ $\sigma_{e_m}\!\in\!S_{e_m}$ such that the coefficient of the monomial $x_{r(e_1)}\cdot x_{r(e_2)} \cdot \ldots \cdot x_{r(e_{m})}$ in the polynomial $p_{\sigma_{e_1} \sigma_{e_2} \ldots \sigma_{e_m}}$ is nonzero.
Finally, let us note that by inequality (\ref{eq-ed-minf}) we obtain that the highest power in the monomial $x_{r(e_1)}\cdot x_{r(e_2)} \cdot \ldots \cdot x_{r(e_{m})}$ is bounded from above by
\begin{align*}
\max _{i\in[n]} |r^{-1}(i)|& = \max _{i\in[n]} |\{j \in [m]: r(e_j) = i \}| \\
&= \max _{i\in[n]} |\{j \in [m]: \min (f_j) = i \}| \leqslant 2 \cdot \lceil {\rm ed}(H) \rceil,
\end{align*}
which completes the proof.
\end{proof}

Another argument can also be presented instead of Lemma \ref{lem:main}. It gives a proof of Theorem \ref{thm:main} that is somewhat shorter but less constructive.
 
Let $H=(V,E)$ be a hypergraph, where $E=\{e_1, e_2,\ldots,e_m\}$ and $|e_i|\geqslant 2$ for each $i=1,2,\ldots,m$. 
For a hypergraph polynomial
$$
p(x_1, \ldots, x_n)=\prod_{i\in[m]} \big( \sum_{v\,\in\,e_i} a_{e_i,v}\,x_{v} \big) \in \F[x_1, \ldots, x_n]
$$
consider the following two families of polynomials in $\F[x_1, \ldots, x_n]$:
\begin{itemize}
\item
$\mathcal{P}(p)$ will be the family of all permuted polynomials 
$$
 p_{\sigma_{e_1}\sigma_{e_2}\ldots\sigma_{e_m}}(x_1, \ldots, x_n)=\prod_{i\in[m]} \big( \sum_{v\,\in\,e_i} a_{e_i,\sigma_{e_i}(v)}\,x_{v} \big) \in \F[x_1, \ldots, x_n],
$$
ranging over all possible choices of permutations $\sigma_{e_1}\!\in S_{e_1},$ $\sigma_{e_2}\!\in S_{e_2},$ $\ldots,$ $\sigma_{e_m}\!\in S_{e_m}$;
\item 
$\mathcal{Q}$ will be the family of all multigraph polynomials
$$
 q_{f_1f_2\ldots f_m}(x_1, \ldots, x_n)=\prod_{i\in[m]} \big( x_{v_i} - x_{w_i} \big) \in \F[x_1, \ldots, x_n],
$$
ranging over all possible choices of two-element subsets $f_i = \{v_i, w_i\} \subseteq e_i$.
\end{itemize}
Note that the family $\mathcal{P}(p)$ depends on the polynomial $p$, while $\mathcal{Q}$ only depends on $H$.

\begin{lemma}\label{lem:comp}
For every hypergraph $H=(V,E)$ and fully unbalanced hypergraph polynomial $p$ of $H$ 
 $$
 \min \{ AT(p') : p' \in \mathcal{P}(p) \} \leq 
 \min \{ AT(q) : q \in \mathcal{Q} \}.
 $$
\end{lemma}
 
\begin{proof}
Assume to the contrary that $\min \{ AT(q) : q \in \mathcal{Q} \} =k$
and $\min \{ AT(p') : p' \in \mathcal{P}(p) \} > k$, i.e. there exists a polynomial $q \in \mathcal{Q}$ containing a monomial
$c \cdot x_1^{\alpha_1}\cdots x_n^{\alpha_n}$ with $c \neq 0$.
and all the exponents satisfying $\alpha_i < k$, but no monomial of that type appears in any element of $\mathcal{P}(p)$.
 
That is, however, not possible because of the following
 
\begin{claim}\label{claim2}
$\mathcal{Q} \subseteq \mbox{span}(\mathcal{P}(p)),$ ie. every element of $\mathcal{Q}$ is a linear combination of polynomials from $\mathcal{P}(p)$.
\end{claim}
 
\begin{proof}[Proof of Claim \ref{claim2}.] Firstly, notice that the statement of the claim is true if the hypergraph has only one edge $e$, i.e., every polynomial of the form
 $x_{u} - x_{w}$ where $u, w \in e$
 is a linear combination of polynomials of the form
 $\sum_{v \in e} a_{e,\sigma(v)} x_{v}$. Indeed, if $a_{e,u} \neq a_{e,w}$,
 then setting $\sigma$ to be the transposition of the vertices $u$ and $w$, we see that
 $$
 x_{u} - x_{w} = \frac{\sum_{v \in e} a_{e,v} x_{v} - \sum_{v \in e} a_{e,\sigma(v)} x_{v}}{a_{e,u} - a_{e,w}}.
 $$
 On the other hand, if $a_{e,u} = a_{e,w} = a$, then by the fact that the polynomial is fully unbalanced, we can find a vertex $v \in e$ for which
 $a_{e,v} \neq a$. As a result, the previous argument can be applied to the 
 pairs $u,v$ and $v,w$, and the polynomial
 $$
 x_{u} - x_{w} = (x_{u} - x_{v}) + (x_{v} - x_{w})
 $$
 is a linear combination of polynomials from $\mathcal{P}(p)$ as the sum of two
 such combinations.
 \end{proof}
 
 Finally, observe that if every polynomial from $\mathcal{Q}$ is a linear combination of polynomials
 from $\mathcal{P}(p)$ in the single edge case, then the same is true when
 we multiply those polynomials corresponding to all edges of the hypergraph. 
\end{proof}

Theorem \ref{thm:main} follows now directly from Lemmas \ref{lem:H-G} and \ref{lem:comp}. Indeed, given a hypergraph $H=(V,E)$ with density not exceeding an integer $k$, we construct by Lemma \ref{lem:H-G} a multigraph $G$ with density not exceeding $k$, whose polynomial $q$ satisfies $AT(q) \leq 2k+1$ by Theorem \ref{thm:max-e}. Invoking Lemma \ref{lem:comp} gives now the conclusion of Theorem \ref{thm:main}.

\section{Applications}\label{sec:app}

As in the case of graphs, the Alon--Tarsi number of hypergraph polynomials can be used to obtain an upper bound on the list chromatic number and the online chromatic number of hypergraphs. Even though the bound provided by Theorem \ref{thm:main} can be simply obtained from Corollary \ref{cor:deg}, we next discuss a direct derivation from the Alon--Tarsi number of a hypergraph polynomial. Let $H=([n],E)$ be a hypergraph where the set of edges contains no singletons. For every edge $e=\{i_1, \ldots, i_k\}\in E$ and some linear ordering of the elements in $E$, consider the linear form
$$
L_e(x_1, \ldots, x_n)=x_{i_1}+\cdots +x_{i_{k-1}}-(k-1)x_{i_k}.
$$
Consider the hypergraph polynomial in $\C[x_1, \ldots, x_n]$:
$$
p_H(x_1, \ldots, x_n)=\prod_{e\in E} L_e(x_1,\ldots,x_n).
$$
This is a fully unbalanced hypergraph polynomial. Up to some reordering of the vertices in each edge of $H$, Theorem \ref{thm:main} ensures that 
$$
AT(p_H)\le 2\lceil {\rm ed}(H)\rceil +1.
$$
Let $s\ge 2$ be a positive integer and denote by $\mathbb{U}_s\subseteq \C$ the multiplicative group of the $s$-th roots of unity. We observe that, for $(a_1, \ldots, a_n)\in \mathbb{U}_s^n$, 
$$
p_H(a_1, \ldots, a_n)=0
$$
if and only if there exists an edge
$$
e = \{i_1, \ldots, i_k\} \in E
$$
such that
$$
a_{i_1} = \ldots = a_{i_k}.
$$
Indeed, for every subset $\{i_1, \ldots, i_k\}\subseteq [n]$ the norm $|a_{i_1}+\cdots +a_{i_{k-1}}|\le k-1$, with equality if and only if $a_{i_1}=\cdots =a_{i_{k-1}}$, in which case, if $e=\{i_1, \ldots, i_k\}$ is an edge of $H$, then $L_e(a_1,\ldots a_n)=0$ if and only if $a_{i_1}=\cdots =a_{i_{k-1}}=a_{i_k}$.

Let $k=AT(p_H)$. By the Combinatorial Nullstellensatz, by choosing $s=k$, there is a choice of $(a_1, \ldots, a_n)\in\mathbb{U}_k^n$ such that $p_H(a_1, \ldots, a_n)\neq 0$. It follows that the coloring $\chi (i)=a_i, i=1, \ldots, n$ produces no monochromatic edges and $\chi (H)\le k$. Similarly, if $A_1, \ldots, A_n$ is a set of lists, each with cardinality at least $k$, then by choosing $s=|A_1\cup \cdots \cup A_n|$ the same argument shows that the list chromatic number $\chi_L(H)$ of $H$ is at most $k$. 

The bound on both numbers can be tight. For instance, the Fano plane has edge density one and chromatic number $3$, while the Alon--Tarsi number of the hypergraph polynomial defined above is upper bounded by $3$. In particular, the list chromatic number of the Fano plane is also three.

A bound for the paintability number of a hypergraph can be obtained from the Alon--Tarsi number as well. In the case of graphs, the respective relation was proved for the first time by Schauz \cite{SchauzPaintCorrect}, whose proof was presented in a simplified form by Zhu and Balakrishnan \cite[Section 3.9]{ZhuBal}. Recently, a different approach was used by Grytczuk, Jendro{\v{l}}, Zaj{\k{a}}c in \cite{GrytczukJZ}. To describe analogous connections for hypergraphs, we shall now describe in some detail the concept of hypergraph paintability. The exposition below follows very closely that of \cite{GrytczukJZ}, with the necessary adjustments.

\begin{definition}
Let $H=(V,E)$ be a hypergraph and let $f:V \rightarrow \mathbb{N} \cup \{0\}$.
We call $H$ \emph{$f$-paintable} if the following conditions hold:
\begin{itemize}
 \item $\forall_{v \in V} f(v) > 0$,
 \item for every nonempty $X \subseteq V$ there exists an independent subset $X' \subseteq X$ (i.e. such that $e\not\subseteq X'$ for every hyperedge $e \in E $) for which the hypergraph $H-X'$ is $f'$-paintable, where the function 
 $f':V \setminus X' \rightarrow \mathbb{N} \cup \{0\}$
 is defined by $f'(v)=f(v)-1$ if $v \in X$ and $f'(v)=f(v)$ if $v \not\in X$.
\end{itemize} 
\end{definition}

Note that the above definition is formally correct because it always refers to a simpler instance of itself,
namely either $H-X'$ has fewer vertices than $H$ or $\sum_{v \in V} f'(v) < \sum_{v \in V} f(v)$ (when $X'=\emptyset$) .
The base case of paintability is $V= \emptyset$,
while the base case of non-paintability is $\exists_{v \in V} f(v) = 0$.

The definition describes formally one round of Schauz's game \cite{Schauz} -- in each
round one player (Lister) reveals a new color in the lists of the vertices belonging to $X$, and the other (Painter)
chooses the subset $X' \subseteq X$ of the vertices that will retain this color
permanently, under the condition that no edge can be monochromatic. Painter wins
if the set of uncolored vertices eventually becomes empty, and loses 
if at some moment an uncolored vertex $v$ has $f(v)=0$, meaning that no new color 
will ever be available for $v$.

Let us also call a hypergraph polynomial of $H=(V,E)$ (we denote $n=|V|$)
$$
p(x_1, \ldots, x_n)=\prod_{e\,\in\, E} \big( \sum_{i\,\in\,e} a_{e,i}\,x_{i} \big) \in \mathbb{F}[x_1, \ldots, x_n]
$$
a {\it coloring polynomial} if and only if
$$
\forall_{e \in E} \sum_{i \in e} a_{e,i} = 0.
$$
The notion comes from the fact that the polynomial defined above takes the value $0$ if any edge is monochromatic. Indeed, if it is colored $c$, then the corresponding expression in brackets takes the value $\sum_{i \in e} c\cdot a_{e,i} = 0$.

Grytczuk, Jendro{\v{l}} and Zaj{\k{a}}c proved the version of the following theorem for graphs in \cite{GrytczukJZ}. Schauz proved its generalization for hypergraphs in \cite{S2010}. The presented proof is a simple generalization of the version in \cite{GrytczukJZ}, but we allow it to be presented for completeness and because it is an alternative argument to the one in \cite{S2010}.

\begin{theorem}\label{thm:paintability}
If a coloring polynomial of the hypergraph $H=(V,E)$ 
$$
p(x_1, \ldots, x_n)=\prod_{e\,\in\, E} \big( \sum_{i\,\in\,e} a_{e,i}\,x_{i} \big) \in \mathbb{F}[x_1, \ldots, x_n],
$$
where $\mathbb{F}$ is an arbitrary field, contains a nonzero monomial of multidegree $(\alpha_1,\ldots, \alpha_n)$, then $H$ is $f$-paintable for $f(x_i)=\alpha_i+1$.
\end{theorem}

The theorem is a direct inductive consequence of the following.

\begin{lemma}
Let $X=\{x_1,\ldots,x_k\} \subseteq V$ and let $V \setminus X=\{y_1,\ldots,y_l\},$ where $k \geqslant 1, l \geqslant 0.$ If $p_H$ contains a nonzero monomial of multidegree $(\alpha_1,\ldots, \alpha_k, \beta_1,\ldots,$ $ \beta_l),$ then there exists an independent subset $X' \subseteq X,$ (to simplify the notation we may renumber the elements of $X$ and assume that $X'=\{x_1,\ldots,x_j\}, j \geqslant 0$), for which $p_{H-X'}$ contains a non-zero monomial of multidegree at most $(\alpha_{j+1}-1,\ldots, \alpha_k-1, \beta_1,\ldots, \beta_l).$
\end{lemma}

\begin{proof}
First, let us define projection operators of the following two types acting on the vector space  $\mathbb{F} [x_1,\ldots,x_k,y_1,\ldots,y_l]$:
\begin{itemize}
 \item $\Pi_{x_i}^{<\alpha_i}$, whose image is spanned by the polynomials $M$
 satisfying $\deg_{x_i}(M)<\alpha_i$, and kernel by those $N$
 for which $\deg_{x_i}(N) \geqslant \alpha_i$;
 \item $\Pi_{y_i}^{\beta_i}$, whose image is spanned by the polynomials $M$
 satisfying $\deg_{y_i}(M) = \beta_i$, and kernel by those $N$
 for which $\deg_{y_i}(N) \not= \beta_i$.
\end{itemize}

Let us now consider the following composition:
$$
T=(I-\Pi_{x_1}^{<\alpha_1}) \circ \ldots \circ (I-\Pi_{x_k}^{<\alpha_k})
\circ\Pi_{y_1}^{\beta_1}\circ
\ldots \circ\Pi_{y_l}^{\beta_l}.
$$
We easily see that in $T(p)$ only the monomials $M$
satisfying $\deg_{x_i}(M) \geqslant \alpha_i$ and $\deg_{y_i}(M) = \beta_i$
for all $i$ will remain unannihilated. But as $p$ is homogeneous, 
there is only one such monomial, namely the one of multidegree $(\alpha_1,\ldots, \alpha_k,
\beta_1,\ldots, \beta_l),$ which does not vanish by the assumption.
Consequently, substituting all values $x_i$ and $y_i$ equal to $1$ will yield a non-zero
constant:
$$
(I-\Pi_{x_1}^{<\alpha_1}) \circ \ldots \circ (I-\Pi_{x_k}^{<\alpha_k})
\circ\Pi_{y_1}^{\beta_1}\circ
\ldots \circ\Pi_{y_l}^{\beta_l} (p)(1,\ldots,1) \in \mathbb{F} \setminus \{0\}.
$$

On the other hand, the above expression is, by linearity, a combination
of $2^k$ expressions of the form
$$
\Pi_{x_{i_1}}^{<\alpha_{i_1}} \circ \ldots \circ \Pi_{x_{i_t}}^{<\alpha_{i_t}}
\circ\Pi_{y_1}^{\beta_1}\circ
\ldots \circ\Pi_{y_l}^{\beta_l} (p)(1,\ldots,1),
$$
corresponding to all $2^k$ subsets $\{i_1,\ldots ,i_t\}$ of $\{1,\ldots,k\}$.


 At least one of them
has therefore to be nonzero, hence, up to renumbering,  at least one of
$$
\Pi_{x_{j+1}}^{<\alpha_{j+1}} \circ \ldots \circ \Pi_{x_k}^{<\alpha_k}
\circ\Pi_{y_1}^{\beta_1}\circ
\ldots \circ\Pi_{y_l}^{\beta_l} (p)
$$
has to be a nonzero polynomial. It follows that $p$ has a non-vanishing monomial of degrees
at most $(\alpha_{j+1}-1,\ldots, \alpha_k-1)$ in the variables 
$(x_{j+1},\ldots, x_k)$, exactly 
$(\beta_1,\ldots, \beta_l)$ in $(y_1,\ldots, y_l)$,
and arbitrary in $(x_1,\ldots,x_j)$.

Let us now see that $X'=\{x_1, \ldots,x_j\}$ satisfies the conditions of the lemma:
\begin{itemize}
\item $p_{H-X'}$  has a monomial of degree at most $(\alpha_{j+1}-1,\ldots, \alpha_k-1, \beta_1,\ldots, \beta_l)$: indeed, $p_H$ can be factorized as $p_H=p_{H-X'}q_H$, where $q_H$ is the product of all linear factors  of $p$ corresponding to the hyperedges in $E$ that have at least one vertex in $X'$;  as $p$ has a monomial of degrees at most $(\alpha_{j+1}-1,\ldots, \alpha_k-1, \beta_1,\ldots, \beta_l)$ in the variables $(x_{j+1},\ldots, x_k,y_1,\ldots, y_l)$, the same is true for $p_{H-X'}$ by the very definition of polynomial multiplication.
\item the set $X'$ is independent: indeed, if there is a hyperedge $e \subseteq X'$, then $p$ contains the factor $\sum_{i\in e} a_{e,i} x_{i}$ which by the assumption $\sum_{i \in e} a_{e,i} =0$ implies $$p(1,1,\dots,1,x_{j+1},\dots,x_k,y_1,\dots,y_l)$$ is identically equal to $0$ as a polynomial of $n-j$ variables, contradicting
$$
\Pi_{x_{j+1}}^{<\alpha_{j+1}} \circ \ldots \circ \Pi_{x_k}^{<\alpha_k} \circ\Pi_{y_1}^{\beta_1}\circ \ldots \circ\Pi_{y_l}^{\beta_l} (p)(1,\ldots,1) \in \mathbb{F} \setminus \{0\},
$$
because the operator of evaluation at $x_i=1$ commutes with both $\Pi_{x_j}^{<\alpha_j}$ for $j \not=i$ and $\Pi_{y_j}^{\beta_j}$ .
\end{itemize}
\end{proof}

As mentioned in the Introduction, as it was communicated to us by Grytczuk \cite{GrytczukPersonal}, the extension of Theorem \ref{thm:main} to every hypergraph polynomial stated in Conjecture \ref{conj:1} provides a simple proof of the 1-2-3 Conjecture of Karo{\'n}ski, {\L}uczak and Thomason. The 1-2-3 Conjecture, now a theorem proved by Keusch \cite{K2023}, states that every connected graph $G=(V, E)$ different from $K_2$ admits an assignment of weights in $\{1,2,3\}$ to the edges such that no two neighbors receive the same sum of incident edge weights. In other words, giving to every vertex the sum of edge weights of its incident edges results in a proper coloring of the graph. 

In order to prove the 1-2-3 Conjecture from Conjecture \ref{conj:1}, given a connected graph $G=(V,E)$ having at least $3$ vertices we may consider the hypergraph $H(G)=(V',E')$ where $V'=E=\{e_1, \ldots, e_m\}$ and, for each edge $e=uv\in E$, there is a (hyper)edge $f_e\in E'$ consisting of all edges incident to $u$ and to $v$ except $e$ (i.e., containing precisely all the edges from the open neighborhood of $e$ in the line graph $L(G)$).

We then consider the hypergraph polynomial 
$$
p_{H(G)}(x_{e_1}, \ldots, x_{e_m})=\prod_{uv\in E} \left(\sum_{e_i\in f_e:u\in e_i} x_{e_i}-\sum_{e_j\in f_e:v\in e_j}x_{e_j}\right).
$$
Thus, $p_{H(G)}$ is nonzero at $(a_1, \ldots, a_m)$ if and only if the assignment of weights $w(e_i)=a_i$ satisfies the requirement that no two neighbors receive the same sum of incident edge weights. We next show that $H(G)$ has edge density at most one for every connected graph $G$ except trees of a special kind, which will be considered separately.

We shall use the standard notation $N(v)=\{w \in V : vw \in E\}$ for the open neighborhood of a vertex $v \in V$ and denote further $N(X)=\bigcup_{v \in X} N(v)$.

 \begin{proposition}\label{prop-graph1}
 Let $H=(V,E)$ be a graph without vertices of degree 0 or 1. If there exists a subset $X \subseteq V$ for which $|N(X)|<|X|$, then $H$ contains an induced claw, i.e. the graph $K_{1,3}$.
 \end{proposition}
 
 \begin{proof}
  Let $Y=X \setminus N(X), Z=N(X) \setminus X$. We easily see that 
 $$
 |Z|=|N(X)|-|N(X) \cap X| < |X|-|N(X) \cap X| = |Y|.
 $$
 Moreover, for every edge $vw \in E$ we have
 \begin{align}
 \label{eq-graph1}
 v \in Y \Leftrightarrow (v \in X \wedge v \not\in N(X)) \Rightarrow
 (w \in N(X) \wedge w \not\in X) \Leftrightarrow w \in Z.
 \end{align}
 Consequently, the set $F=\{yz \in E : y \in Y \wedge z \in Z\}$ has at least $2|Y|$ elements, because by the assumption every $y \in Y$ has at least two neighbors in $V(H)$ and by (\ref{eq-graph1}) they all belong to $Z$.
 
 On the other hand, the inequality $|F| \geq 2|Y| > 2|Z|$ implies by the pigeonhole principle that there exists a vertex $z \in Z$ adjacent to at least three distinct vertices $y_1, y_2, y_3 \in Y$. Thus, an induced claw is formed, because again by (\ref{eq-graph1})
 no two of the vertices $y_1, y_2, y_3$ can be adjacent since $Y \cap Z = \emptyset$.
 
 \end{proof}
 
The following comes immediately.

\begin{corollary}\label{cor-graph1}
If $G$ is a graph whose line graph does not have vertices of degree $0$ or $1$, then there exists a bijection $f : E(G) \rightarrow E(G)$ with the property that for every $e \in E(G)$ the edges $e$ and $f(e)$ have exactly one vertex in common, i.e. the corresponding vertices in $L(G)$ are adjacent.
\end{corollary}
 
 \begin{proof}
 We apply Proposition \ref{prop-graph1} to the graph $H=L(G)$. As no line graph can contain an induced claw, we deduce that for all subsets $X \subseteq V(H)$ the inequality $|N_H(X)| \geq |X|$ holds. The statement now follows from Hall's theorem.
 \end{proof}

 Let us also analyze what happens when the graph $L(G)$ contains vertices of degree $0$ or $1$. Assuming that $G$ is connected, the former is only possible when $G \simeq K_2$, and this graph is explicitly excluded from the statement of the 1-2-3 Conjecture.
 
 On the other hand $\deg_{L(G)}(uv)=1 \Leftrightarrow \deg_G(u)+\deg_G(v)=3$, meaning that one of the degrees (without loss of generality $\deg(u)$) equals 1, the only neighbor of $u$ is $v$, and $v$ has one neighbor other than $u$. Let us introduce the name {\it 2-pendant} to describe the structure consisting of a vertex of degree 1, named $u$, a vertex of degree 2, named $v$, the edge $uv$, and one more edge attaching $v$ to some vertex $w \in V(G) \setminus \{u,v\}$. 
 
 Note that a graph $G$ with a 2-pendant on vertices $u$ and $v$ has a bijection described in Corollary \ref{cor-graph1} if and only if the graph $G - u - v$ obtained by removing the 2-pendant has such a bijection. The reason is that in the required bijection of $E(G)$, the edges $uv$ and $vw$ have to be mapped onto each other as $uv$ has in $L(G)$ no neighbor other than $vw$.
 
 We may therefore conclude that Corollary \ref{cor-graph1} is also true for many graphs with 2-pendants, namely the graphs in which the iterative procedure of removing one 2-pendant after another until no 2-pendants exist ends in a graph different from $K_2$. This is in particular always the case when $G$ is connected and not a tree, because removing 2-pendants does not destroy any cycles existing in $G$.

Note also that a graph $G$ has a perfect matching if and only if the graph $G-u-v$ resulting from removing the $2$-pendant on vertices $u$ and $v$ has a perfect matching $M$ and the perfect matching in $G$ is precisely $M\cup \{uv\}$. Thus, finally, one can observe that the above results cannot be applied only to the trees with perfect matching. This allows us to formulate the following.

\begin{corollary}
For every connected graph $G=(V,E)$, not being a tree with a perfect matching the hypergraph $H(G)=(V^\prime,E^\prime)$ has edge density $ed(H(G))\le 1$.
\end{corollary}

\begin{proof}
 Let $f:E\rightarrow E$ be a bijection guaranteed by the Corollary \ref{cor-graph1}. There is a bijection $g:E \rightarrow E^\prime$ defined in a natural way with $g(e)=f_e$. Consider their composition $h=g\circ f$. It is also a bijection such that $h:E\rightarrow E^\prime$ and for every $e\in E$, $e\in h(e)$.

Now, consider any subset of the vertices of $H(G)$, say $X\subseteq V^\prime=E$. The bijection $h$ matches every hyperedge in $E^\prime(X)$ with a distinct vertex $e\in X$, thus $h^{-1}(E^\prime(X))\subseteq X$ and consequently $|E^\prime(X)|=|h^{-1}(E^\prime(X))|\leq |X|$, which completes the proof.
\end{proof}

Given a graph $G=(V,E)\neq K_2$, let us define the following polynomial, where the variables are assigned to the edges:
$$
p_G(x_{e_1}, \ldots, x_{e_m})=\prod_{uv\in E} \left(\sum_{e_i\ni u} x_{e_i}-\sum_{e_j\ni v}x_{e_j}\right).
$$
 It is straightforward to see, that when putting some labels $a_{e_1},a_{e_2},\dots,a_{e_m}$ on the edges, the obtained vertex coloring induced by the sums of labels on the edges incident to the vertices is proper if and only if $p_G(a_{e_1},a_{e_2},\dots,a_{e_m})\neq 0$. The above considerations imply that if Conjecture \ref{conj:1} holds true, then $AT(p_G)\leq 3$ for every graph $G$ not being a tree with a perfect matching. Actually, it would be enough to prove Conjecture \ref{conj:1} for hypergraph polynomials with coefficients $\pm 1$. On the other hand, Bartnicki, Grytczuk, and Niwczyk \cite[Corollary 3]{BGN} proved in particular that $AT(p_T)\leq 3$ for every tree $T\neq K_2$. Thus Conjecture \ref{conj:1} implies Conjecture 3 from \cite{BGN} stating that $AT(p_G)\leq 3$ for every connected graph $G\neq K_2$. According to our previous considerations, it further implies the choosability version of the 1-2-3 Conjecture for every graph $G\neq K_2$ (even its paintability version if, in addition, Conjecture \ref{conj:PCN}, stated at the end of this section, holds true) and, in consequence, the 1-2-3 Conjecture itself. As already mentioned, the latter one was recently proved by Keusch \cite{K2023}. On the other hand, to our knowledge, the best bound for the list version of the 1-2-3 Conjecture is five colors, as proved by Zhu \cite{Z2022}.

An interesting generalization of the 1-2-3 conjecture worth mentioning is the paintable version, with the following definition.

\begin{definition}
Let $G=(V, E)$ be a graph, let $w: E_0 \rightarrow \mathbb{F}$ be 
a function from the subset of edges $E_0 \subseteq E$ to the set of elements of some field $\mathbb{F}$ and let $A\subseteq\mathbb{F}$ be its arbitrary subset. Moreover, let $f:E\setminus E_0 \rightarrow \mathbb{N} \cup \{0\}$.
We say that the ordered pair $(G,w)$ is \emph{$(f,A)$-unbalanceable} if the following conditions hold:
\begin{itemize}
\item $A\neq\emptyset$,
\item $f(e)>0$ for every $e\in E \setminus E_0$,
\item if $E_0=E$ (i.e. $w: E \rightarrow \mathbb{F}$), then the weighting $w$ satisfies the condition stated in the 1-2-3 Conjecture, that is, for any two adjacent vertices $v,v' \in V$ we have $ \sum_{e\ni v} w(e) \neq \sum_{e'\ni v'} w(e')$,
\item if $E_0 \subsetneq E$, then for every nonempty $X \subseteq E\setminus E_0$ and any member $a\in A$ there exists a subset $X' \subseteq X$ for which the pair $(G,w')$ is $(f',A \setminus\{a\})$-unbalanceable, where
\begin{itemize}[label=$\circ$]
\item $w':E_0 \cup X' \rightarrow \mathbb{F}$ is defined by $w'(e)=w(e)$ if $e \in E_0$ and $w'(e)=a$ if $e \in X'$,
\item $f':E \ \setminus \ (E_0 \cup X') \rightarrow \mathbb{N} \cup \{0\}$ is defined by $f'(e)=f(e)-1$ if $e \in X$ and $f'(e)=f(e)$ if $e \not\in X$.
\end{itemize}
\end{itemize}
\end{definition}

Note that the given definition is correct in the sense that it refers to a simpler instance of itself, because either $X'\neq \emptyset$ or $\sum_e f'(e)<\sum_e f(e)$.

\begin{definition}
We say that graph $G=(V,E)$ is \emph{$f$-unbalanceable} if $(G,\emptyset)$ is $(f,A)$-unbalanceable for arbitrarily chosen infinite set $A\subseteq\mathbb{F}$, where $\emptyset$ denotes the empty function, i.e., the function with empty domain.
\end{definition}

Using the above definitions, one can propose the following intriguing twin of Theorem \ref{thm:paintability} (let us call it \emph{playable 1-2-3 Conjecture}).

\begin{restatable}{restatableconjecture}{conPlayable} \label{conj:PCN}
Let $G=(V, E)$ be a graph without isolated edges. Then $G$ is $f$-unbalanceable where $f:E\rightarrow \mathbb{N}\cup\{0\}$ is defined by $f(e)=3$ for every $e\in E$.
\end{restatable}

Grytczuk \cite{GrytczukPersonal} pointed out that Conjecture \ref{conj:1} implies also two other famous conjectures. The first of them is Kahn’s Permanent Conjecture \cite{deVos, NagyPach, Yu}. Recall that the permanent rank $perrank(A)$ of a $n\times n$ matrix $A$ is the size of the largest square submatrix of $A$ with nonzero permanent \cite{Yu}, and $AA$ denotes the $n\times 2n$ matrix consisting of two copies of $A$. The Kahn's Conjecture reads as follows.
\begin{conjecture}\label{con:kahn}
For any $n\times n$ nonsingular matrix $A$ over any field, $perrank(AA)=n$.
\end{conjecture}

Let us show that it follows indeed from Conjecture \ref{conj:1}. Consider an arbitrary nonsingular matrix $A$ of size $n\times n$. We associate each row $e$ of $A$ with an edge and each column $i$ of $A$ with a vertex of a hypergraph $H$, where $i\in e\Leftrightarrow a_{e,i}\neq 0$, where $a_{e,i}$ is the entry of $A$ corresponding with edge $e$ and vertex $i$. Let the hypergraph polynomial of $H$ be
$$
p_H(x_1,\ldots,x_{n}) = \prod_{e\in H}(\sum_{i\in e} a_{e,i}x_i).
$$
We claim that $ed(H)\leq 1$. Indeed, if there were a subset of $V(H)$ of size $k$ inducing at least $k+1$ edges (note that $k<n$), they would correspond with a $k\times(k+1)$ submatrix of $A$ with the rows being obviously linearly dependent. Since the remaining entries of these rows in $A$ would be all $0$'s, this would imply the existence of $k+1\leq n$ linearly dependent rows of $A$, and contradict its non-singularity. Since $ed(H)\leq 1$ and Conjecture \ref{conj:1} is assumed to be true, we obtain that $AT(p_H)\leq 3$. This means, in turn, that in the expansion of $p_H$ there is a non-vanishing monomial $M_\alpha = c_{\alpha}x_1^{\alpha_1}\cdots x_n^{\alpha_n}$ with $\max \,\{\alpha_1, \ldots, \alpha_n\} \leq 2$. It is a well-known fact \cite{AT1989} that $c_\alpha \prod_{j=1}^n (\alpha_j!) = per(A^\prime)$, where $A^\prime$ is any matrix consisting of selected columns of $A$, where each column $i$ is taken exactly $\alpha_i$ times. This means in particular that $c_\alpha\neq 0$ if and only if there exists a submatrix of $AA$ of size $n\times n$ having nonzero permanent, i.e., if and only if $perrank(AA)=n$.

It is also known (see e.g. \cite{Yu}) that Conjecture \ref{con:kahn}, in turn, implies the Alon--Jaeger--Tarsi Conjecture \cite{AT1989, Yu}, which has been recently solved for sufficiently large primes by Nagy and Pach \cite{NagyPach}.

\begin{conjecture}[Alon--Jaeger--Tarsi Conjecture]
For any field $F$ with $|F|\geq 4$ and any nonsingular matrix $A$ over $F$, there is a vector $x$ such that both $x$ and $Ax$ have only nonzero entries.
\end{conjecture}

\section{Remarks and open problems}\label{sec:final}

Recall that in our opinion, the most valuable achievement would be to prove the following hypothesis, implying in particular an immediate proof of the list version of the 1-2-3 Conjecture.

\conMain*

Even the following weaker version of Conjecture \ref{conj:1} would be a significant contribution.

\begin{conjecture}\label{conj:weak}
For every hypergraph $H=(V,E)$ and fully unbalanced hypergraph polynomial $p$ of $H$ we have
$$
AT(p) \leqslant 2 \cdot \lceil {\rm ed}(H) \rceil +1.
$$
\end{conjecture}

Another significant result would be to prove the playable 1-2-3 Conjecture.
\conPlayable*

Since, to our knowledge, there is no constant bound known for the online version of the 1-2-3 Conjecture problem, even the following relaxation seems interesting.

\begin{conjecture}
Let $G=(V, E)$ be a graph without isolated edges. Then there exists a constant $c\in \mathbb{N}\cup\{0\}$ such that $G$ is $f$-unbalanceable where $f:E\rightarrow \mathbb{N}\cup\{0\}$ is defined by $f(e)=c$ for every $e\in E$.
\end{conjecture}

\end{document}